\documentclass{amsart}
\usepackage{amssymb,amscd,amsmath,mathabx,enumerate,verbatim,calc}
\usepackage{amscd,amssymb,amsopn,amsmath,amsthm,graphics,amsfonts,enumerate,verbatim,calc}
\usepackage{mathrsfs}
\usepackage[dvips]{graphicx}
\usepackage[autostyle]{csquotes}
\usepackage[colorlinks=true,linkcolor=red,citecolor=blue]{hyperref}
\usepackage{mathrsfs}
\usepackage{tikz-cd}
\usepackage{tikz}
\usetikzlibrary{matrix,arrows,decorations.pathmorphing}
\input xy
\xyoption{all}
\calclayout

\makeatletter
\newcommand{\xRrightarrow}[2][]{\ext@arrow 0359\Rrightarrowfill@{#1}{#2}}
\newcommand{\Rrightarrowfill@}{\arrowfill@\equiv\equiv\Rrightarrow}
\newcommand{\xLleftarrow}[2][]{\ext@arrow 3095\Lleftarrowfill@{#1}{#2}}
\newcommand{\Lleftarrowfill@}{\arrowfill@\Lleftarrow\equiv\equiv}
\newcommand{\xLleftRrightarrow}[2][]{\ext@arrow 3399\LleftRrightarrowfill@{#1}{#2}}
\newcommand{\LleftRrightarrowfill@}{\arrowfill@\Lleftarrow\equiv\Rrightarrow}
\makeatother

\newcommand{\rt}{\rightarrow}

\newcommand{\st}{\stackrel}

\newcommand{\n}{\frak{n}}

\newcommand{\CX}{\mathcal{X} }

\newcommand{\im}{{\rm{Im}}}

\newcommand{\add}{{\rm{add}}}

\newcommand{\gen}{{\rm{gen}}}
\newcommand{\cogen}{{\rm{cogen}}}

\newcommand{\RNum}[1]{\uppercase\expandafter{\romannumeral #1\relax}}

\newcommand{\pd}{{\rm{pd}}}
\newcommand{\id}{{\rm{id}}}

\newcommand{\Coker}{{\rm{Coker}}}
\newcommand{\Ker}{{\rm{Ker}}}

\newcommand{\Hop}{{\rm{op}}}

\newcommand{\Tor}{{\rm{Tor}}}
\newcommand{\Hom}{{\rm{Hom}}}
\newcommand{\Ext}{{\rm{Ext}}}
\newcommand{\End}{{\rm{End}}}

\newtheorem{theorem}{Theorem}[section]
\newtheorem{corollary}[theorem]{Corollary}
\newtheorem{lemma}[theorem]{Lemma}

\theoremstyle{definition}
\newtheorem{definition}[theorem]{Definition}
\newtheorem{example}[theorem]{Example}

\newtheorem{remark}[theorem]{Remark}

\theoremstyle{plain}

\theoremstyle{definition}

\numberwithin{equation}{section}

\begin{document}

\title[On the Wakamatsu tilting conjecture]{On the Wakamatsu tilting conjecture}

\author[K. Divaani-Aazar, A. Mahin Fallah and M. Tousi]{Kamran Divaani-Aazar, Ali Mahin Fallah and Massoud Tousi}

\address{K. Divaani-Aazar, Department of Mathematics, Faculty of Mathematical Sciences, Alzahra University,
Tehran, Iran.}
\email {Email address: kdivaani@ipm.ir}

\address{A. Mahin Fallah, School of Mathematics, Institute for Research in Fundamental Sciences (IPM), P.O.
Box: 19395-5746,
Tehran, Iran.}
\email{amfallah@ipm.ir, ali.mahinfallah@gmail.com }

\address{M. Tousi, Department of Mathematics, Faculty of Mathematical Sciences, Shahid Beheshti University,
Tehran, Iran.}
\email{mtousi@ipm.ir}

\subjclass[2020]{16D20, 16G99, 18G05, 18E40.}

\keywords{Artin algebra, perfect ring, tilting module, Wakamatsu tilting module.\\
The research of the second author is supported by a grant from IPM (No.1402160015).}

\begin{abstract} Let $R$ be an associative ring with identity. We establish that the generalized Auslander-Reiten conjecture
implies the Wakamatsu tilting conjecture. Furthermore, we prove that any Wakamatsu tilting $R$-module of finite projective
dimension that is tensorly faithful is projective. By utilizing this result, we show the validity of the Wakamatsu tilting
conjecture for $R$ in two cases: when $R$ is a left Artinian local ring or when it is the group ring of a finite group $G$
over a commutative Artinian ring.
\end{abstract}

\maketitle

\section{Introduction}

This paper concentrates on the Wakamatsu tilting conjecture. Tilting theory plays a remarkable role in the representation
theory of Artin algebras. The original concept of tilting modules was introduced by Brenner and Butler in \cite{BB}, and
Happel and Ringel in \cite{HR} for finite-dimensional algebras. Miyashita \cite{M} extended the classical notion of tilting
modules to encompass finitely presented modules of finite projective dimension over arbitrary rings.

Wakamatsu \cite{W2} made further generalizations to the concept of tilting modules over Artin algebras, allowing for the
possibility of infinite projective dimension. These generalized tilting modules are commonly referred to as Wakamatsu
tilting modules, following the established terminology in \cite{GRS}. In \cite{W1}, the definition of Wakamatsu tilting
modules was extended to apply to any ring, leading to the natural generalization of many results on tilting modules over
Artin algebras.

Let $R$ be an Artin algebra. The following homological conjectures play a significant role in the representation
theory of Artin algebras.

\vspace{0.1cm}
{\bf Finitistic Dimension Conjecture (FDC):}  The supremum of the projective dimension of all finitely generated
left $R$-modules with finite projective dimension is finite.

\vspace{0.1cm}
{\bf Wakamatsu Tilting Conjecture (WTC):} If $T_R$ is a Wakamatsu tilting $R$-module with finite projective dimension,
then $T$ is a tilting $R$-module.

\vspace{0.1cm}
{\bf Gorenstein Symmetry Conjecture (GSC):} $\id(_RR)<\infty$ if and only if $\id(R_R)<\infty$.

\vspace{0.1cm}
{\bf Generalized Auslander-Reiten Conjecture (GARC):} Let $M$ be a finitely generated left $R$-module and $n$ a
non-negative integer. If $\Ext_R^i(M,M\oplus R)=0$ for all $i>n$, then $\pd(_RM)\leq n$.

\vspace{0.1cm}
{\bf  Auslander-Reiten Conjecture (ARC):} Let $M$ be a finitely generated left $R$-module. If $\Ext_R^i(M,M\oplus R)=0$
for all $i>0$, then $M$ is projective.

\vspace{0.1cm}
{\bf Generalized Nakayama Conjecture (GNC):} Every indecomposable injective left $R$-module appears (up to isomorphism)
as a direct summand of some term of the minimal injective resolution of $_RR$.

\vspace{0.1cm}
{\bf Nakayama Conjecture (NC):} If $0\rt R\rt I^0\rt I^1\rt \cdots$ is a minimal injective resolution of $_RR$, with all
$I^i$ being projective, then $R$ is a self-injective algebra.

\vspace{0.1cm}

Nakayama \cite{N} proposed the seventh conjecture in 1958. The sixth and fifth conjectures were proposed by Auslander and
Reiten \cite{AR2}. Additionally, they proposed the third conjecture \cite{AR1}. In 2010, Wei \cite{We3} put forward the
fourth conjecture. The Wakamatsu tilting conjecture was proposed by Beligiannis and Reiten \cite[Chapter III]{BR}. Lastly,
Bass \cite{Bas} introduced the first conjecture.

It is known that the above conjectures exhibit a significant interrelation. To illustrate this further, we show that the
generalized Auslander-Reiten conjecture implies the Wakamatsu tilting conjecture. More precisely, we establish the following
result:

\begin{theorem}\label{1.1} Let $_ST_R$ be a Wakamatsu tilting module with $\pd(T_R)<\infty$. If the generalized Auslander-Reiten
conjecture holds for $S$, then $T_R$ is a tilting module.
\end{theorem}

Wei \cite{We2} proved that the (GARC) holds for Artin algebras in which every module possesses an ultimately closed projective
resolution. Consequently, (GARC) is valid for every Artin algebra that is of finite representation type, torsionless-finite,
or radical square zero.

Considering the previously established connections between the above seven conjectures and Theorem \ref{1.1}, we can represent
the implication diagram for these conjectures as follows:

\vspace{0.5cm}

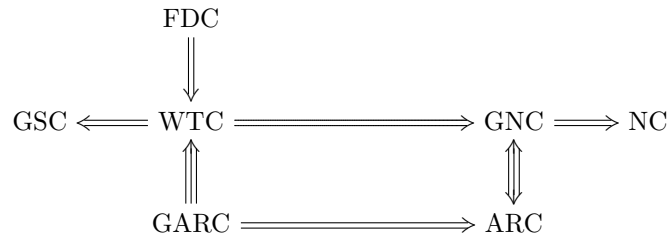
\begin{center}
$\xymatrix{
\text{} & \text{FDC} \ar@{=>}[d]\\
\text{GSC} \ar@{<=}[r] & \text{WTC} \ar@//@{=>}[rrr]  &&& \text{GNC} \ar@{=>}[r] & \text{NC}\\
&\text{GARC}  \ar@3{->}[u]  \ar@2{->}[rrr]  &&& \text{ARC}  \ar@3{<->}[u]
}$
\end{center}
(Here, the notation (p)$\Rightarrow$(q) means that every Artin algebra satisfying (p) also satisfies (q), while (p)$\Rrightarrow$(q)
means that if all Artin algebras satisfy (p), then they also satisfy (q).)

\vspace{0.7cm}

Mantese and Reiten \cite{MR} proved the implication (FDC)$\Rightarrow$(WTC). Beligiannis and Reiten established the implications
(WTC)$\Rightarrow$(GNC) and (WTC)$\Rightarrow$(GSC); see \cite[Proposition 3.1]{BR}. Auslander and Reiten proved the implication
(GNC)$\Rightarrow$(NC) and the equivalence (GNC)$\xLleftRrightarrow{}$(ARC) in \cite{AR2}. The implication (GARC)$\Rightarrow$(ARC)
is trivial. Finally, the implication (GARC)$\Rrightarrow$(WTC) holds by Theorem \ref{1.1}. In the papers \cite[Lemma 6.6(3)]{CPX}
and \cite[Example 3.7]{D}, the authors noticed that an example constructed by Schulz in 1994 \cite{Sc} serves as a counterexample
to (GARC). Furthermore, \cite[Remark 3.8]{D} demonstrates that (GARC) and (ARC) are not equivalent.

Mantese and Reiten \cite{MR} have proved that the Wakamatsu tilting conjecture holds for Gorenstein Artin algebras. The Wakamatsu
tilting conjecture is closely linked to other homological conjectures, as discussed in \cite{H, We2, We1, Z, En}. Recently, Cruz
\cite{C} showed that the Wakamatsu tilting conjecture is true for all finite-dimensional algebras over a field if and only if,
over any such algebra, every self-orthogonal $n$-quasi-generator has projective dimension at most $n$.

According to the setting of \cite{W1}, we consider the Wakamatsu tilting conjecture for arbitrary rings. Let $R$ be an associative
ring with identity. We prove that:

\begin{theorem}\label{1.2} Let $T_R$ be a Wakamatsu tilting $R$-module of finite projective dimension. If $T$ is tensorly
faithful, then $T$ is projective.
\end{theorem}

Furthermore, we establish that if $R$ is left perfect and left Kasch, then every Wakamatsu tilting module is tensorly faithful.
This allows us to deduce the main result of this paper:

\begin{corollary}\label{1.3} The Wakamatsu tilting conjecture holds for $R$ if $R$ is either a left Artinian local ring or the group
ring of a finite group $G$ over a commutative Artinian ring.
\end{corollary}

\section{The Results}

Throughout this article, the word {\it ring} refers to an associative ring with identity, and all modules are unitary. Let $R$
and $S$ be two rings. To indicate that an $R$-module $M$ is a right module (respectively, a left module), we use the symbol $M_R$
(respectively, $_RM$).

Let $\mathscr{C}$ be a full subcategory of the category of left $S$-modules. The notation $\mathscr{C}^{\perp}$ denotes the subcategory
of $S$-modules $_SN$ such that $\Ext^{i\geqslant1}_{S}(X,N)=0$ for all $X\in \mathscr{C}$. Dually, the notation $~^{\perp}\mathscr{C}$
denotes the subcategory of $S$-modules $_SN$ such that $\Ext^{i\geqslant1}_{S}(N,X)=0$ for all $X\in \mathscr{C}$.

Let $_ST_R$ be a bimodule. We denote by $\add(_ST)$ the class of left $S$-modules which are isomorphic to a direct summand of
a direct sum of finitely many copies of $T$. Let $\gen^*(T)$ denote the class of $S$-modules $_SN$ for which, there exists an
exact sequence $$\cdots \st{f_2} \longrightarrow T_1 \st{f_1}\longrightarrow T_0 \st{f_0}\longrightarrow N\longrightarrow 0$$
with each $T_i\in \add(_ST)$ and $\Ext^1_S(T,\Ker f_i)=0$ for all $i\geqslant 0$. We put $\CX_{{T}}:={T}^\perp \cap \gen^*(T)$.
Dually,  let $\cogen^*(T)$ denote the class of $R$-modules $N_R$ for which there exists an exact sequence of the form
$$0\rt N \st{f_0}\rt T_0 \st{f_1}\rt T_1 \st{f_2}\rt \cdots$$ with each $T_i\in \add(T_R)$ and $\Ext^1_R(\Coker f_i,T)=0$ for all
$i\geqslant 0$.

Let us start by considering the following two definitions.

\begin{definition}\label{2.1} (See \cite{M}.) Let $T_R$ be an $R$-module and $n\in \mathbb{N}$. We say that $T_R$ is $n$-{\it tilting}
if
\begin{itemize}
\item[(i)]  $T_R\in \gen^*(R)$ and $\pd(T_R)\leq n$,
\item[(ii)] $\Ext^i_{R}(T,T)=0$ for all $i>0$, and
\item[(iii)] There is an exact sequence $$0\rt R\rt T_0 \rt T_1\rt \cdots \rt T_n\rt 0,$$ where $T_i\in \add(T_R)$ for all
$0\leq i\leq n$.
\end{itemize}
\end{definition}

\begin{definition}\label{2.2} (See \cite[Section 3]{W1}.) Let $T_R$ be an $R$-module. We say that $T_R$ is a {\it Wakamatsu tilting}
module if
\begin{itemize}
\item[(i)] $T_R\in \gen^*(R)$,	
\item[(ii)] $\Ext^i_{R}(T,T)=0$ for all $i>0$, and
\item[(iii)] There is an exact sequence $$0\rt R \st{f_0}\rt T_0 \st{f_1}\rt T_1\st{f_2}\rt T_2\rt \cdots ,$$ where $T_i\in
\add(T_R)$ for all $i$, which stays exact after applying the functor $\Hom_R(-,T)$.
\end{itemize}
\end{definition}

It is easy to verify that every $n$-tilting module is also a Wakamatsu tilting module. The concept of Wakamatsu tilting left
modules is defined in a similar manner.

By \cite[Corollary 3.2]{W1}, we have the following characterization of the Wakamatsu tilting modules.

\begin{lemma}\label{2.3} For a bimodule $_ST_R$, the following are equivalent:
\begin{itemize}
\item[(i)] $T_R$ is a Wakamatsu tilting module with $S\cong \End(T_R)^{\Hop};$
\item[(ii)] $_ST$ is a Wakamatsu tilting module with $R\cong \End(_ST)^{\Hop};$
\item[(iii)] One has
\begin{itemize}
\item[(1)]  $T_R\in \gen^*(R)$  and $~_ST\in \gen^*(S)$.
\item[(2)]  $R\cong \End(_ST)^{\Hop}$ and $S\cong \End(T_R)^{\Hop}$.
\item[(3)]  $\Ext^i_{R}(T,T)=0$  and $\Ext^i_{S}(T,T)=0$ for all $i>0$.
\end{itemize}
\end{itemize}
\end{lemma}

Henceforth, a bimodule $_ST_R$ that satisfies the conditions stated in Lemma \ref{2.3} will be referred to as a Wakamatsu tilting
module.

Now, we are prepared to provide the proof of Theorem \ref{1.1}.

\begin{theorem}\label{2.18} Let $_ST_R$ be a Wakamatsu tilting module with $\pd(T_R)<\infty$. If the generalized Auslander-Reiten
conjecture holds for $S$, then $T_R$ is a tilting module.
\end{theorem}

\begin{proof} Assume that $\pd(T_R)=n$. So, there exists an exact sequence $$0\rt P_n \rt \cdots \rt P_2 \rt P_1\rt P_0\rt
T\rt 0, \quad (*)$$ in which each $P_i$ is a finitely generated projective right $R$-module. Since $\Hom_R(T,T)\cong S$ and
$\Ext^i_R(T,T)=0$ for all $i>0$, applying the functor $\Hom_R(-,T)$ to $(*)$ gives rise to the exact sequence $$0\rt S\rt
\Hom_R(P_0,T)\rt \Hom_R(P_1,T)\rt \cdots \rt \Hom_R(P_n,T) \rt 0.$$  It is evident that each $\Hom_R(P_i,T)$ belongs to
$\mathrm{add}(_ST)$. Hence, we have an exact sequence $$0\rt S\rt T_0\rt T_1  \rt \cdots \rt T_n \rt 0,$$  where $T_i\in
\mathrm{add}(_ST)$ for all $i=0,1, \dots, n$. Since $T_i\in \mathrm{add}(_ST)$ and $\Ext^{j>0}_S(T,T)=0$, it follows that
$\Ext^{j>0}_S(T,T_i)=0$ for all $i=0,1, \dots, n$. Hence, $$0\rt T_0\rt T_1  \rt \cdots \rt T_n\rt
0$$ is a $\Hom_S(T,-)$-cyclic coresolution of $_SS$. Consequently, for each non-negative integer $j$, $\Ext^j_S(T,S)$ is
the $j$ th homology of the following complex: $$0\rt \Hom_S(T,T_0)\rt \Hom_S(T,T_1)  \rt \cdots \rt \Hom_S(T,T_n)\rt 0.$$
In particular, $\Ext^j_S(T,S)=0$ for all $j>n$. Thus, $$\Ext^j_S(T,T\oplus S)\cong \Ext^j_S(T,T)\oplus \Ext^j_S(T,S)=0$$
for all $j>n$. Therefore, due to the validity of the generalized Auslander-Reiten conjecture for $S$, it follows that
$\pd(_ST)$ is finite.

Since $\pd(_ST)<\infty$, the module $_ST$ admits a finite  projective resolution $$0 \rt P_{\ell} \overset{d_{\ell}}\rt P_{\ell-1}
\rt \cdots \rt P_0\overset{d_0}\rt	T\rt 0, \quad (*)$$ where each $P_i$ is a finitely generated projective left $S$-module. Since
$\Hom_S(T,T)=R$ and $\Ext^i_S(T,T)=0$ for all $i>0$, applying the functor $\Hom_S(-,T)$ to $(*)$ gives rise to the exact sequence
$$0\rt R\rt \Hom_S(P_0,T)\rt \Hom_S(P_1,T)\rt \cdots \rt \Hom_S(P_{\ell},T) \rt 0.$$  It is evident that $\Hom_S(P_i,T)\in
\add(T_R)$ for all $1\leq i\leq \ell$. Hence, we get an exact sequence  $$0\rt R\rt T_0 \rt T_1\rt \cdots \rt T_{\ell}\rt
0,$$ where $T_i\in \add(T_R)$ for all $0\leq i\leq \ell$. By \cite[Proposition 3.5]{Baz}, we may assume that $\ell=n$. Therefore,
$T_R$ is a tilting module, as required.
\end{proof}

According to the definition given by Holm and White \cite[Definition 2.1]{HW}, a bimodule $_ST_R$ is referred to as a
semidualizing $(S,R)$-bimodule if it satisfies the conditions outlined in Lemma \ref{2.3}(iii). Hence, Lemma \ref{2.3}
asserts that a bimodule $_ST_R$ is a Wakamatsu tilting module if and only if it is a semidualizing $(S,R)$-bimodule.

To present our next result, we need to recall the definition of Bass classes.

\begin{definition}\label{2.4} Let $_SC_R$ be a semidualizing $(S,R)$-bimodule. The {\it Bass class} $\mathscr{B}_C(S)$ is the
class of all $S$-modules $_SM$ for which the evaluation map $$\nu_M^C:C\otimes_R\Hom_S(C,M)\rt M$$ is an isomorphism
and $$\Ext^i_S(C,M)=0=\Tor^R_i(C,\Hom_S(C,M))$$ for all $i\geq 1$.
\end{definition}

The next result is already known for Artin algebras; see \cite[Lemma 3.12]{WX}.

\begin{lemma}\label{2.5} Let $T_R$ be a Wakamatsu tilting module with $\End(T_R)=S$. Then $$\CX_{T}=\mathscr{B}_{T}(S)\cap
\gen^*(S).$$
\end{lemma}

\begin{proof} See \cite[Corollary 3.2]{Su}.
\end{proof}

Next, let us recall the definition of tensorly faithfulness.

\begin{definition}\label{2.6} An $R$-module $T_R$ is referred to as \textit{tensorly faithful} if the functor $T\otimes_R-$
is faithful, meaning that $T\otimes_RM$ is nonzero for every nonzero $R$-module $_RM$.
\end{definition}

\begin{remark}\label{2.61} Given $T_R\in \gen^*(R)$, \cite[Lemmas 3.1 and 1.2(a)]{HW} imply that $T_R$ is tensorly faithful
if and only if $\Hom_R(T,X)\neq 0$ for all nonzero $R$-modules $X$. For a projective $R$-module $T_R$ in $\gen^*(R)$, let
$\lambda:T \otimes_R \Hom(T, R)\to R$ be the trace map and $I=\im \ \lambda$. The observation that $\Hom_R(T,R/I)=0$ leads
to the conclusion that $T_R$ is tensorly faithful if and only if it is a generator for the category of right $R$-modules.
\end{remark}

The concept of tensorly faithful modules was originally introduced in \cite{HW} under the name \enquote{faithfully modules}.
In order to prevent confusion with faithful modules, we renamed this notion.

Note that if $R$ is a commutative Noetherian ring, and $S=R$, then by \cite[Proposition 3.1]{HW} every Wakamatsu tilting $R$-module
is tensorly faithful. Next, we present an example of a non-commutative ring $R$ and a Wakamatsu tilting $R$-module that is tensorly
faithful.

\begin{example}\label{2.7} Let $\Gamma$ be a commutative Noetherian ring. Suppose that $Q$ is an acyclic quiver. It is fairly easy
to verify that path algebra $R=\Gamma Q$ is a free $\Gamma$-algebra; see \cite[Definition 2.2.5]{S}. Let $E$ be a Wakamatsu tilting
$\Gamma$-module. Then \cite[Proposition 3.2]{HW} yields that $T_R=E\otimes_{\Gamma}R$ is a tensorly faithful Wakamatsu tilting
$R$-module.
\end{example}

To establish Theorem \ref{1.2}, we require the following two lemmas. The proof of the first one is closely connected
to \cite[Theorem 6.3]{HW}, where the Bass class $\mathscr{B}_T(S)$ is shown to be closed under kernels of epimorphisms.

\begin{lemma}\label{2.8} Let $_ST_R$ be a Wakamatsu tilting module. If $T_R$ is tensorly faithful, then the class $\CX_{T}$ is closed
under kernels of epimorphisms.
\end{lemma}

\begin{proof} Assuming that $T_R$ is tensorly faithful, let us consider a short exact sequence of left $S$-modules:
$$0 \rightarrow M' \rightarrow M \rightarrow M'' \rightarrow 0,$$ where $M$ and $M''$ are in $\mathcal{X}_T$. According to Lemma
\ref{2.5}, we need to prove that $$M'\in \mathscr{B}_T(S)\cap \mathrm{gen}^*(S).$$ By \cite[Lemma 4.3]{W1}, it turns out that $M'\in
\mathrm{gen}^*(S)$, and so it remains to show that $M'\in \mathscr{B}_T(S)$.
	
Consider the following diagram with exact rows:\\

\vspace{0.3cm}
\begin{tikzcd}
T\otimes_R\Hom_S(T,M)\arrow{r}\arrow{d}{\nu_{M}^T}	&T\otimes_R\Hom_S(T,M'')\arrow{r}\arrow{d}{\nu_{M''}^T}
&T\otimes_R\Ext_S^1(T,M')\arrow{r}&0\\
M\arrow{r}&M''\arrow{r}&0
\end{tikzcd}
\vspace{0.3cm}\\
Since the maps $\nu_M^T$ and $\nu_{M''}^T$ are isomorphisms, we deduce $T\otimes_R \mathrm{Ext}_S^1(T,M')=0$.
As $T_R$ is tensorly faithful, we get $\mathrm{Ext}_S^1(T,M')=0$. Additionally, for any $i\geq 1$, we
have the exact sequence $$\mathrm{Ext}_S^i(T,M'') \rightarrow \mathrm{Ext}_S^{i+1}(T,M') \rightarrow
\mathrm{Ext}_S^{i+1}(T,M).$$ Thus, we conclude that $\mathrm{Ext}_S^i(T,M')=0$ for all $i\geq 1$. From the
exact sequence $$0\rightarrow \mathrm{Hom}_S(T,M')\rightarrow \mathrm{Hom}_S(T,M) \rightarrow \mathrm{Hom}_S(T,M'')
\rightarrow 0,$$ and the fact that $M$ and $M''$ belong to $\mathscr{B}_C(S)$, it follows that
$\mathrm{Tor}^R_i(T,\mathrm{Hom}_S(T,M'))=0$ for all $i\geq 1$.

Next, we consider the following diagram with exact rows:

\vspace{0.3cm}
\begin{tikzcd}
0\arrow{r} 	&T\otimes_R\Hom_S(T,M')\arrow{r}\arrow{d}{\nu_{M'}^T}	&T\otimes_R\Hom_S(T,M)
\arrow{r}\arrow{d}{\nu_{M}^T}	&T\otimes_R\Hom_S(T,M'')\arrow{r}\arrow{d}{\nu_{M''}^T} &0\\
0\arrow{r} &M'\arrow{r} &M\arrow{r} &M''\arrow{r} &0
\end{tikzcd}
\vspace{0.3cm}\\
Since the maps $\nu_M^T$ and $\nu_{M''}^T$ are isomorphisms, we deduce that $\nu_{M'}^T$ is also an isomorphism.
Therefore, we have $M' \in \mathscr{B}_T(S)$.
\end{proof}

\begin{lemma}\label{2.9} Let $_ST\in \gen^*(S)$ be such that $\Ext^i_S(T,T)=0$ for all $i>0$. Then $\add(_ST)=\CX_{T}\cap
^{\perp}\CX_T$.
\end{lemma}

\begin{proof}  First, we show that $\add(_ST)\subseteq \CX_T \cap ^\perp \CX_T$. The inclusion $\CX_T \subseteq {T}^\perp$ yields
that $$^\perp({T}^\perp) \subseteq ^\perp \CX_T.$$ Furthermore, since $T \in ^\perp({T}^\perp)$, we can deduce that $\add(_ST)
\subseteq ^\perp({T}^\perp)$. Consequently, $\add(_ST)\subseteq ^\perp \CX_T$. Additionally, it is evident from the definition
that $\add(_ST)\subseteq \CX_T$. Thus, $\add(_ST)\subseteq \CX_T \cap ^\perp \CX_T$.
	
Conversely, suppose that $N\in \CX_{T} \cap^{\perp}\CX_T$. Since $N\in \CX_{T}$, there exists an exact sequence $$0\rt N_1\rt T'
\rt N\rt 0,$$ where $T'\in \add(_ST)$ and $N_1\in \CX_{T}$. Since $N\in~^{\perp}\CX_T$ and $N_1\in \CX_{T}$, we get
$\Ext_R^1(N,N_1)=0$. Thus the above exact sequence splits, and so $N\in \add(_ST)$.
\end{proof}

Now, we are ready to prove Theorem \ref{1.2}.

\begin{theorem}\label{2.10} Let $_ST_R$ be a Wakamatsu tilting module. Assume that $\pd(T_R)<\infty$ and $T_R$ is tensorly faithful.
Then $T_R$ is a projective tilting module.
\end{theorem}

\begin{proof} Suppose that $\pd(T_R)=n$. Then, as in the proof of Theorem \ref{2.18},  there exists an exact sequence $$0\rt S\rt
T_0\rt T_1  \rt \cdots \rt T_n \rt 0,$$  where $T_i\in \mathrm{add}(_ST)$ for all $i=0,1, \dots, n$. Thus,
according to Lemma \ref{2.9}, we conclude that $T_i\in \mathcal{X}_T$ for all $i=0,1, \dots, n$. Consequently, by Lemma
\ref{2.8}, $S\in \mathcal{X}_T$. This implies the existence of an integer $\ell>0$ and a split short exact sequence:
$$0 \rightarrow V \rightarrow T^{\ell} \rightarrow S \rightarrow 0.$$ As a result, $\Hom_S(S,T)$ becomes a direct summand of the
right $R$-module $\Hom_S(T^{\ell},T)$. Therefore, $T_R$ is projective.
\end{proof}

In light of the above result, we can provide the following example of a non-tensorly faithful Wakamatsu tilting $R$-module.

\begin{example}\label{2.12} Let $\Gamma$ be an Artinian commutative ring with identity and $R=\Gamma{\bf{Q}}$ be the
Artin $\Gamma$-algebra that corresponds to the quiver ${\bf{Q}}:1\leftarrow 2 \leftarrow 3$; see \cite[Chapter VI,
Example 1.2(d)]{ASS}. It is easy to see that $T:=P(1)\oplus P(3)\oplus I(3)$ is a 1-tilting non-projective $R$-module.
By Theorem \ref{2.10}, $T$ is not tensorly faithful.
\end{example}

\begin{definition}\label{2.13} (See \cite{F,L}.)
\begin{itemize}
\item[(i)] A ring $R$ is called left {\it Kasch} if for every left maximal ideal $\n$ of $R$, there exists a left ideal
$I$ of $R$ that is isomorphic to $R/{\n}$.
\item[(ii)] A ring $R$ is called a left {\it max\ ring} if every nonzero left $R$-module has a maximal submodule.
\end{itemize}
\end{definition}

Let $R$ be a ring, and let $E(R)$ denote the injective envelope of the $R$-module $_RR$. The ring $R$ is left Kasch if and 
only if $E(R)$ serves as a cogenerator for the category of left $R$-modules. Furthermore, $R$ is left Kasch if and only if 
$\Hom_R(M,R)\neq 0$ for every nonzero finitely generated $R$-module $_RM$. (For these assertions, see \cite[Proposition 
1.44]{NY}.) It is easy to see that every commutative Artinian ring is left Kasch. Also, we can see that every left semisimple 
ring is Kasch.

According to \cite[Theorem 28.4]{AF}, a ring $R$ is left perfect if and only if the ring $R/J(R)$ is semisimple and every 
nonzero left $R$-module has a maximal submodule. Hence, every left perfect ring is a left max ring. Additionally, a ring $R$ 
is left perfect if and only if every descending chain of principal left ideals in $R$ terminates. Thus, every left Artinian 
ring is left perfect and, consequently, left max.

\begin{lemma}\label{2.14} Let $_ST$ be an $S$-module such that the ring $R=\End(_ST)$ is left max and Kasch. Then $T_R$ is
tensorly faithful.
\end{lemma}

\begin{proof} On the contrary, suppose that there exists a nonzero left $R$-module $N$ such that $T\otimes_RN=0$.  As $R$ is a left
max ring, $N$ has a maximal submodule $K$. Set $M=N/K$. Since the functor $T\otimes_R-$ is right exact and $T\otimes_RN=0$, we get
$T\otimes_RM=0$. Using this, we have:
$$
\begin{array}{lllll}
0&=\Hom_S(T\otimes_RM,T)\\
&\cong \Hom_R(M,\Hom_S(T,T))\\
&\cong \Hom_R(M,R).
\end{array}
$$
	
Since $R$ is left Kasch, it has a left ideal $I$ that is isomorphic to $M$. Applying the left exact functor $\Hom_R(M,-)$ to the inclusion
$0\rightarrow I\rightarrow R$ yields the exact sequence $$0\rightarrow \Hom_R(M,I)\rightarrow \Hom_R(M,R).$$ Thus, $$\Hom_R(M,M)\cong
\Hom_R(M,I)=0.$$ However, since $M$ is nonzero, it follows that $\Hom_R(M,M)\neq 0$, which leads to a contradiction.
\end{proof}

In conclusion, we would like to make the following remark. The results presented in parts (iii) and (iv) are Corollary \ref{1.3}.

\begin{remark}\label{2.15} The Wakamatsu tilting conjecture is valid for $R$ under any of the following conditions:
\begin{itemize}
\item[(i)]  $R$ is a left max and Kasch ring.
\item[(ii)] $R$ is a left Artinian ring with zero Jacobson radical.
\item[(iii)] $R$ is a left Artinian local ring.
\item[(iv)] $R$ is the group ring of a finite group $G$ over a commutative Artinian ring.
\end{itemize}
\end{remark}

\begin{proof}
(i) It follows immediately by Lemma \ref{2.14} and Theorem \ref{2.10}.

(ii) Since $R$ is a left Artinian ring with zero Jacobson radical, it follows that $R$ is semisimple. Consequently, every
short exact sequence of left $R$-modules splits. This implies that $R$ is left Kasch. Now, as any left Artinian ring is
left max, (i) completes the proof.

(iii) It is known that every left Artinian ring is left max. Consider the unique left maximal ideal $\frak m$ of
$R$. Let $I$ be a minimal nonzero left ideal of $R$. Since $I$ is simple, it is isomorphic to $R/\frak m$. Hence, $R$
is left Kasch, and the proof is completed by (i).

(iv) Let $A$ be a commutative Artinian ring and $R=A[G]$. Since $R$ is finitely generated as an $A$-module, it follows that
$R$ is a left Artinian ring, and so it is left max. Additionally, as $A$ is left Kasch, \cite[Theorem 2.4]{Sh} implies that
$R$ is also left Kasch. Thus, the conclusion follows by (i).
\end{proof}

\section*{Acknowledgement}

We would like to thank the anonymous referee for their careful and thorough reading of the paper, as well as their
valuable comments, particularly for suggesting Remark 2.8.


\end{document}